\documentclass[11pt,letterpaper]{article}
\usepackage[T1]{fontenc}
\usepackage{lmodern,amsmath,amssymb,amsthm,mathtools,microtype}
\usepackage[margin=0.85in]{geometry}
\usepackage{tikz}
\usetikzlibrary{arrows.meta,positioning,fit,backgrounds}
\usepackage{enumitem,booktabs,array,caption,needspace,fancyhdr}
\usepackage[colorlinks=true,linkcolor=blue!45!black,urlcolor=blue!45!black,citecolor=blue!45!black]{hyperref}
\definecolor{ink}{RGB}{28,65,100}
\definecolor{accent}{RGB}{178,75,28}
\definecolor{pale}{RGB}{231,241,249}
\newcommand{\ES}{Erd\H{o}s--S\'os}
\newcommand{\dd}{\overline d}
\newcommand{\E}{\mathbb E}

\newcommand{\Gcut}[2]{G\big|_{(#1,#2)}}
\newtheorem{theorem}{Theorem}
\newtheorem{proposition}{Proposition}
\newtheorem{corollary}{Corollary}
\newtheorem*{ESrestatement}{Theorem 1 (restated)}
\theoremstyle{definition}\newtheorem{definition}{Definition}
\setlist{itemsep=3pt,topsep=5pt}
\hypersetup{
  pdftitle={The \ES\ Conjecture},
  pdfauthor={Jay Cummings},
  pdfsubject={A visual exposition of the proof discovered by GPT-6 Astra},
  pdfkeywords={Erdos-Sos theorem, tree embeddings, average degree, extremal graph theory}
}
\tikzset{v/.style={circle,draw=ink,fill=white,inner sep=1pt,minimum size=5.5mm,font=\small},
 edge/.style={draw=black!30,line width=.7pt},
 chosen/.style={draw=ink,line width=1.7pt},
 root/.style={v,line width=1.3pt,fill=pale},
 arrival/.style={v,draw=accent,line width=1pt,fill=accent!12}}
\newcommand{\graphpic}[3]{%
\begin{tikzpicture}[x=.85cm,y=.85cm,baseline=(current bounding box.center)]
\foreach \name/\xx/\yy in {#1}{\coordinate (\name) at (\xx,\yy);}
\foreach \a/\b in {#2}{\draw[edge] (\a)--(\b);}
\foreach \a/\b in {#3}{\draw[chosen] (\a)--(\b);}
\foreach \name/\xx/\yy in {#1}{\node[v] at (\name){$\name$};}
\end{tikzpicture}}
\newcommand{\hostpic}[3]{%
\begin{tikzpicture}[x=.85cm,y=.85cm,baseline=(current bounding box.center)]
\path[use as bounding box] (-.35,-.35) rectangle (2.35,2.35);
\coordinate (a) at (0,1); \coordinate (b) at (1,2);
\coordinate (c) at (2,1); \coordinate (d) at (0,0);
\coordinate (e) at (2,0);
\foreach \aa/\bb in {#2}{\draw[edge] (\aa)--(\bb);}
\foreach \aa/\bb in {#3}{\draw[chosen] (\aa)--(\bb);}
\foreach \name/\sty in {#1}{\node[\sty] at (\name){$\name$};}
\end{tikzpicture}}
\newcommand{\panel}[2]{\begin{minipage}[c]{.47\linewidth}\centering #1\par\smallskip\small #2\end{minipage}}
\newcommand{\threepanel}[2]{\begin{minipage}[c]{.31\linewidth}\centering #1\par\smallskip\small #2\end{minipage}}
\newenvironment{pictureblock}{\par\addvspace{8pt}\noindent\begin{minipage}{\linewidth}\centering}{\end{minipage}\par\addvspace{8pt}}

\title{\textbf{The \ES\ Conjecture}\\[6pt]
\large A Visual Exposition of the Proof Discovered by GPT-6 Astra}
\author{Jay Cummings\\[-1pt]
\small Department of Mathematics and Statistics, California State University, Sacramento}
\date{25 September 2026}
\begin{document}
\maketitle

\begin{abstract}
The Erd\H{o}s--S\'os theorem states that every graph of average degree greater than $t-2$ contains every tree on $t$ vertices. A short counting proof was discovered by GPT-6 Astra in 2026. We give a visual, reader-centered exposition of that argument whose presentation differs substantially from the original. We recast the counting objects as \emph{early neighbors} in a reveal-and-stop procedure, organize the induction through explicit partitions and reversible swaps, and develop the proof through worked examples with consistent drawings. We also give direct and probabilistic conclusions, compare this formulation with other recent expositions, and record the classical consequence $R(T;q)\le q(t-2)+2$ for multicolor Ramsey numbers of trees.
\end{abstract}



\section{A Little Notation}
If you are unfamiliar with basic graph theory terminology and ideas, see Appendix~\ref{app:graphs}. Here we quickly note that $V(G)$ and $E(G)$ are the vertex and edge sets of $G$, $n=|V(G)|$ is its number of vertices, $e(G)=|E(G)|$ is its number of edges, and $\deg_G(v)$ is the degree of $v$. We also write $\deg(v)$ when the graph is understood. All graphs in this paper are finite, simple, and undirected; whenever average degree is used, we assume $n\ge1$. Recall that the handshaking identity says
\begin{equation}
\sum_{v\in V(G)}\deg_G(v)=2e(G),
\label{eq:handshake}
\end{equation}
because each edge contributes one to the degree of each of its two endpoints. Hence the average degree is
\begin{equation}
\dd(G)=\frac1n\sum_{v\in V(G)}\deg_G(v)=\frac{2e(G)}n.
\label{eq:average}
\end{equation}

\Needspace{6\baselineskip}
\section{The Question: Enough Edges to Contain Every Tree}
The \ES\ conjecture asks how many edges force a graph to contain every tree of a given size. Throughout, \emph{contains} means contains as a subgraph, not necessarily an induced subgraph: unused vertices and extra edges are allowed.

\begin{theorem}[The \ES\ statement]\label{thm:ES}
Let $n\ge t\ge2$. If $G$ is a finite simple graph on $n$ vertices with $\dd(G)>t-2$ (which is equivalent to saying that $e(G)>\frac{(t-2)n}{2}$), then $G$ contains every tree on $t$ vertices.
\end{theorem}

For large $t$, there are many different shapes that a tree $T$ can take on $t$ vertices. What this conjecture is saying is that as long as $G$ has enough edges, every single one of these tree shapes exists within $G$ as a subgraph.

The strict inequality matters. For example, recall that $G$ need not be connected, and consider a $G$ built in this way: Take disjoint copies of the complete graph $K_{t-1}$, the graph with every possible edge on $t-1$ vertices. Every vertex has degree $t-2$, but every connected component is too small to contain a $t$-vertex tree. Thus equality cannot suffice in general, even for perfectly balanced degrees.
\begin{pictureblock}
\panel{\graphpic{a/0/0,b/1/1.4,c/2/0,d/3/0,e/4/1.4,f/5/0}{a/b,a/c,b/c,d/e,d/f,e/f}{a/b,a/c,b/c,d/e,d/f,e/f}}{For $t=4$: two separate triangles, all degrees $2$.}
\hfill
\panel{\graphpic{a/0/0,b/0/1.6,c/1.6/1.6,d/1.6/0,e/3/0,f/3/1.6,g/4.6/1.6,h/4.6/0}{a/b,a/c,a/d,b/c,b/d,c/d,e/f,e/g,e/h,f/g,f/h,g/h}{a/b,a/c,a/d,b/c,b/d,c/d,e/f,e/g,e/h,f/g,f/h,g/h}}{For $t=5$: two separate copies of $K_4$, all degrees $3$.}
\captionof{figure}{At the threshold, small complete components obstruct every target tree of the next size.}
\end{pictureblock}

\Needspace{6\baselineskip}
\section{Historical Context and the Source of This Proof}\label{sec:history}
Paul Erd\H{o}s and Vera T. S\'os posed this problem in the early 1960s. Chung's survey dates the conjecture to 1962 \cite[Problem~68]{chung}; its standard published reference is Erd\H{o}s's 1964 problem collection \cite{erdos1964}. It became a central problem in extremal graph theory: how many edges force a prescribed subgraph? As quoted by Adamczewski and Bloom, Fan Chung described it as ``one of the most tantalizing problems in extremal graph theory'' \cite[Appendix~B.4]{frontier}.

Some target trees were understood long before the general statement. For a star on $t$ vertices, average degree greater than $t-2$ guarantees a vertex of degree at least $t-1$, which supplies its center and leaves. The path case follows from the classical Erd\H{o}s--Gallai bound on graphs without a long path. Further work established the conjecture for various families of trees, including double stars and other restricted shapes. A modern account of these earlier results is given in \cite{bounded}.

In the early 1990s, Ajtai, Koml\'os, Simonovits, and Szemer\'edi announced a proof for sufficiently large trees \cite{bounded}. The distinction between an announced proof and a published account matters here: the 2023 article \cite{beyond}, and again Reed and Stein's September 2026 preprint \cite{dense}, describe the earlier announcement as lacking a published manuscript.

Progress continued in the intervening years. Besomi, Pavez-Sign\'e, and Stein proved results for bounded-degree trees in dense host graphs \cite{bounded}. Pokrovskiy's 2024 preprint proved the conjecture for sufficiently large trees with a fixed bound on their maximum degree, without requiring a dense host graph \cite{pokrovskiy}. Davoodi, Piguet, \v{R}ada, and Sanhueza-Matamala established an asymptotic version for dense host graphs without bounding the target tree's degree, presented in a 2023 conference paper \cite{beyond} and developed in a March 2026 preprint \cite{asymptotic}.

In September 2026, Reed and Stein posted a proof of the exact conjecture when the target tree has at least a fixed positive fraction of the host graph's vertices and the host is sufficiently large \cite{dense}. They state that their proof was developed without AI and completed before the announcement of Astra's proof; their preprint was posted on 4 September.

The September 2026 \emph{FrontierMath Erd\H{o}s} report by Tom Adamczewski and Thomas F. Bloom attributes a proof of the full conjecture to a pre-release GPT-6 Astra \cite[Appendix~B.4]{frontier}. The original counting proof is available as \cite{astra}, and the accompanying repository records the autonomous proof search and its Lean formalization \cite{lean}.\footnote{The repository's externally compared theorem uses $e(G)\ge(t-2)n/2+1$, a slightly stronger hypothesis than $e(G)>(t-2)n/2$ when $(t-2)n$ is odd. Its internal counting lemma nevertheless gives the sharp classical bound. We use the strict inequality throughout.} The report credits Bloom with the more detailed informal expositions on the corresponding Erd\H{o}s Problems website pages, including his account of this proof \cite{bloom}.

In the introduction to Appendix~B, Adamczewski and Bloom describe those informal expositions as provisional sketches and call for traditional papers by human experts that supply fuller details and context \cite{frontier}. This article seeks to contribute to that task. An earlier four-page illustrated account \cite{illustrated}, which credits Bloom but does not identify its own author, was the starting point for this exposition.

Several recent papers offer other perspectives. Riordan and Scott \cite{riordan-scott} count \emph{jumping edges} and also determine the extremal graphs. Wood \cite{wood} gives a detailed account using ordering--neighbor pairs. Frederickson \cite{frederickson} uses random cyclic orderings and trees with both a specified root vertex and a specified incident edge. Section~\ref{sec:comparisons} explains how these viewpoints relate to ours.

The present article retains Astra's central counting argument, but changes its presentation in three ways: it counts \emph{early neighbors} through a reveal-and-stop procedure; it organizes the induction around explicit partitions and reversible swaps; and it develops the proof through worked examples with consistent drawings, followed by direct and probabilistic conclusions. Its contribution is a reader-centered, visual exposition of the argument, not a new resolution of the conjecture. The additional detail is intended to make clear what is counted, why a rearrangement can be reversed, and where each inequality comes from.

The method has also prompted extensions. Mubayi and Verstra\"ete apply it to tight trees in hypergraphs and to oriented trees in Eulerian digraphs, with the proofs attributed to GPT-6 Astra \cite{kalai,digraphs}. Section~6 of Santos, Stein, and Williams \cite{butterflies} adapts the argument to antidirected trees in digraphs; this is an addition to a paper whose main approach uses regularity methods. Riordan--Scott and Frederickson also prove the antidirected extension. These directed results are distinct from the Eulerian result. We stay with undirected graphs here and conclude with a short classical application to edge colorings in Section~\ref{sec:ramsey}. This historical overview reflects sources checked through 24 September 2026; it is not a comprehensive survey.

\subsection*{Why Average Degree Is a Demanding Hypothesis}
There is an easy statement with \emph{minimum} degree in place of average degree: if $G$ is nonempty and every vertex has degree at least $t-1$, then $G$ contains every $t$-vertex tree. Root the tree anywhere and embed it one vertex at a time, always placing a parent before its children. When $s<t$ vertices have been placed, an already placed parent has at most $s-1$ neighbors among those vertices. Its degree is at least $t-1$, so it has an unused neighbor for the next child.

The difficulty is that an average says little about an individual vertex. Some vertices may have very large degree while others have hardly any neighbors. The \ES\ statement nevertheless guarantees \emph{every} tree shape at the lower threshold $\dd(G)>t-2$. We next examine both balanced and concentrated degree patterns before proving the general result.

\Needspace{6\baselineskip}
\section{The First Example: Trees on Four Vertices}

To better understand this conjecture, let's look at two small cases, $t=4$ and $t=5$. For each, we consider graphs whose degrees are as balanced as possible and graphs that concentrate edges at as many universal vertices as the edge count permits.

For the case that $t=4$, we wish to show that our graph $G$ contains every tree on four vertices. There are, in fact, two such trees: the path $P_4$, and the star $K_{1,3}$. 
\begin{pictureblock}
\panel{\graphpic{a/0/0,b/1/0,c/2/0,d/3/0}{a/b,b/c,c/d}{a/b,b/c,c/d}}{$P_4$: degree sequence $(2,2,1,1)$.}
\hfill
\panel{\graphpic{a/1/1,b/0/0,c/2/0,d/1/2}{a/b,a/c,a/d}{a/b,a/c,a/d}}{$K_{1,3}$: degree sequence $(3,1,1,1)$.}
\captionof{figure}{Both trees must occur whenever $e(G)>n$.}
\end{pictureblock}
\subsection*{Example of a Balanced $G$: A Cycle Plus One Chord}
Start with a cycle on $n\ge4$ vertices. Every degree is $2$, and there are $n$ edges. A path on four consecutive vertices is already present, but a three-leaf star is impossible because no vertex has degree $3$.

Add an edge between two previously nonadjacent cycle vertices, called a \emph{chord}. Now $e(G)=n+1$ and $\dd(G)=2+2/n>2$. The chord's endpoints have degree $3$; all other degrees remain $2$. For example, let $G$ be the following six-vertex graph.
\begin{pictureblock}
\graphpic{a/0/1,b/1/2,c/2/2,d/3/1,e/2/0,f/1/0}{a/b,b/c,c/d,d/e,e/f,f/a,a/d}{a/b,b/c,c/d,d/e,e/f,f/a,a/d}
\captionof{figure}{The host graph $G$: a six-cycle with the chord $ad$. It has seven edges and average degree $7/3>2$.}
\end{pictureblock}
This $G$ contains both trees on four vertices, as shown below. The path uses four consecutive cycle vertices. For the star, the two cycle neighbors of $a$ and its chord neighbor give three leaves. Edges between these neighbors, if present, would be harmless.
\begin{pictureblock}
\panel{\graphpic{a/0/1,b/1/2,c/2/2,d/3/1,e/2/0,f/1/0}{a/b,b/c,c/d,d/e,e/f,f/a,a/d}{a/b,b/c,c/d}}{A path was already there: $a-b-c-d$.}
\hfill
\panel{\graphpic{a/0/1,b/1/2,c/2/2,d/3/1,e/2/0,f/1/0}{a/b,b/c,c/d,d/e,e/f,f/a,a/d}{a/b,a/f,a/d}}{The new chord supplies a third neighbor of $a$.}
\captionof{figure}{The same host graph, with the two required trees highlighted separately.}
\end{pictureblock}
\subsection*{Example of a Concentrated $G$: One Universal Vertex}
A vertex adjacent to every other vertex is called \emph{universal}; it has degree $n-1$. A star on $n$ vertices has one such center, but its $n-1$ edges give average degree $2-2/n$, below the threshold. Add two distinct edges among the leaves. We now have $n+1$ edges and average degree $2+2/n$. For example, let $G$ be the following graph on five vertices.
\begin{pictureblock}
\graphpic{u/1.5/1.5,a/0/0,b/1/0,c/2/0,d/3/0}{u/a,u/b,u/c,u/d,a/b,c/d}{u/a,u/b,u/c,u/d,a/b,c/d}
\captionof{figure}{The host graph $G$: a four-leaf star with the two additional edges $ab$ and $cd$.}
\end{pictureblock}
This $G$ contains both trees on four vertices, as shown below. The three-leaf star remains centered at $u$. To find $P_4$, take the added edge $ab$ and a third leaf $c$; then $b-a-u-c$ is a path. In fact, one added edge already creates this path, but two are required to meet the strict density hypothesis.
\begin{pictureblock}
\panel{\graphpic{u/1.5/1.5,a/0/0,b/1/0,c/2/0,d/3/0}{u/a,u/b,u/c,u/d,a/b,c/d}{u/a,u/b,u/c}}{The center still supplies the star.}
\hfill
\panel{\graphpic{u/1.5/1.5,a/0/0,b/1/0,c/2/0,d/3/0}{u/a,u/b,u/c,u/d,a/b,c/d}{b/a,a/u,u/c}}{An edge among the leaves supplies the path.}
\captionof{figure}{For $n=5$, both panels have $6=n+1$ edges.}
\end{pictureblock}
With $n+1$ edges to place in the graph, one universal vertex is the maximum when $n\ge5$, since two require $2n-3>n+1$ edges. When $n=4$, the resulting five-edge graph has two universal vertices instead. The examples illustrate the theorem; they are not a classification of all graphs above its threshold.

\Needspace{6\baselineskip}
\section{The Next Example: Trees on Five Vertices}
Now there are three shapes: the path $P_5$, the star $K_{1,4}$, and a tree obtained by subdividing one edge of $K_{1,3}$. Call the last tree $F_5$. It has a degree-three vertex with two short arms and one arm of length two.
\begin{pictureblock}
\threepanel{\graphpic{a/0/0,b/1.25/0,c/2.5/0,d/3.75/0,e/5/0}{a/b,b/c,c/d,d/e}{a/b,b/c,c/d,d/e}}{$P_5$\\$(2,2,2,1,1)$}
\hfill
\threepanel{\graphpic{a/1/1,b/0/1,c/2/1,d/1/0,e/1/2}{a/b,a/c,a/d,a/e}{a/b,a/c,a/d,a/e}}{$K_{1,4}$\\$(4,1,1,1,1)$}
\hfill
\threepanel{\graphpic{a/1/1,b/0/2,c/0/0,d/2/1,e/3/1}{a/b,a/c,a/d,d/e}{a/b,a/c,a/d,d/e}}{$F_5$\\$(3,2,1,1,1)$}
\captionof{figure}{All five-vertex trees. The degree sum is $8$ in every case.}
\end{pictureblock}
To check completeness, a degree-four vertex forces the star, and maximum degree two forces a path. In the remaining case, a degree-three vertex leaves exactly one further degree above one, so the degree sequence is $(3,2,1,1,1)$, forcing $F_5$.

The density condition is now $\dd(G)>3$, or $e(G)>3n/2$. A long cycle plus one chord does \emph{not} meet it: its average degree is $2+2/n$, and its maximum degree is only $3$, so it cannot contain $K_{1,4}$.

\subsection*{Balanced Example: A Cubic Graph Plus One Edge}
A \emph{cubic} graph has degree $3$ at every vertex. Adding an edge between two previously nonadjacent vertices gives those two vertices degree $4$ and leaves every other degree at $3$. The average becomes $3+2/n$, just above the threshold. For example, let $G$ be the following graph, obtained from a triangular prism by adding the edge $ae$.
\begin{pictureblock}
\graphpic{a/0/1,b/1.4/2,c/2.8/1,d/0/-1,e/1.4/0,f/2.8/-1}{a/b,b/c,c/a,d/e,e/f,f/d,a/d,b/e,c/f,a/e}{a/b,b/c,c/a,d/e,e/f,f/d,a/d,b/e,c/f,a/e}
\captionof{figure}{The host graph $G$: a cubic graph plus one edge. It has six vertices, ten edges, and average degree $10/3>3$.}
\end{pictureblock}
This $G$ contains all three trees on five vertices:

\begin{pictureblock}
\threepanel{\graphpic{a/0/1,b/1.4/2,c/2.8/1,d/0/-1,e/1.4/0,f/2.8/-1}{a/b,b/c,c/a,d/e,e/f,f/d,a/d,b/e,c/f,a/e}{d/a,a/b,b/c,c/f}}{$P_5$: $d-a-b-c-f$.}
\hfill
\threepanel{\graphpic{a/0/1,b/1.4/2,c/2.8/1,d/0/-1,e/1.4/0,f/2.8/-1}{a/b,b/c,c/a,d/e,e/f,f/d,a/d,b/e,c/f,a/e}{a/b,a/c,a/d,a/e}}{$K_{1,4}$ centered at $a$.}
\hfill
\threepanel{\graphpic{a/0/1,b/1.4/2,c/2.8/1,d/0/-1,e/1.4/0,f/2.8/-1}{a/b,b/c,c/a,d/e,e/f,f/d,a/d,b/e,c/f,a/e}{a/b,a/d,a/c,c/f}}{$F_5$: long arm $a-c-f$.}
\captionof{figure}{Three copies inside the same almost regular graph. Gray edges are present but unused.}
\end{pictureblock}

\subsection*{Concentrated Example: Two Universal Vertices}
Take two adjacent universal vertices $u,v$, each joined to all other vertices, and put no edges between the remaining vertices. This gives $2n-3$ edges and average degree $4-6/n$, exceeding $3$ exactly when $n>6$. For example, let $G$ be the following graph on seven vertices.
\begin{pictureblock}
\graphpic{u/0/1,v/3.2/1,a/0/-1,b/.8/-1,c/1.6/-1,d/2.4/-1,e/3.2/-1}{u/v,u/a,u/b,u/c,u/d,u/e,v/a,v/b,v/c,v/d,v/e}{u/v,u/a,u/b,u/c,u/d,u/e,v/a,v/b,v/c,v/d,v/e}
\captionof{figure}{The host graph $G$: two vertices of degree six and five vertices of degree two.}
\end{pictureblock}
This $G$ contains all three trees on five vertices, as shown below. With distinct remaining vertices $a,b,c,d$, use
\[
P_5:\ a-u-b-v-c;\qquad
K_{1,4}:\ ua,ub,uc,ud;\qquad
F_5:\ ua,ub,uv,vc.
\]
\begin{pictureblock}
\threepanel{\graphpic{u/0/1,v/3.2/1,a/0/-1,b/.8/-1,c/1.6/-1,d/2.4/-1,e/3.2/-1}{u/v,u/a,u/b,u/c,u/d,u/e,v/a,v/b,v/c,v/d,v/e}{a/u,u/b,b/v,v/c}}{An alternating path.}
\hfill
\threepanel{\graphpic{u/0/1,v/3.2/1,a/0/-1,b/.8/-1,c/1.6/-1,d/2.4/-1,e/3.2/-1}{u/v,u/a,u/b,u/c,u/d,u/e,v/a,v/b,v/c,v/d,v/e}{u/a,u/b,u/c,u/d}}{Four neighbors of $u$.}
\hfill
\threepanel{\graphpic{u/0/1,v/3.2/1,a/0/-1,b/.8/-1,c/1.6/-1,d/2.4/-1,e/3.2/-1}{u/v,u/a,u/b,u/c,u/d,u/e,v/a,v/b,v/c,v/d,v/e}{u/a,u/b,u/v,v/c}}{The edge $uv$ joins the arms.}

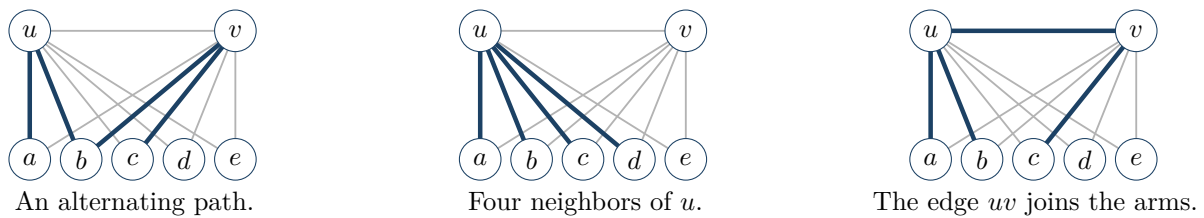
\captionof{figure}{For $n=7$: eleven edges, average degree $22/7>3$.}
\end{pictureblock}
Three universal vertices would require $3n-6>2n-3$ edges for $n>3$, so this model really does maximize their number among graphs with $2n-3$ edges. At $n=5$ or $6$, the same model is below or at the density threshold; one additional edge among the independent vertices makes the inequality strict. The displayed tree constructions still work, using $v$ as a star leaf if needed.

\subsection*{The Same Two Patterns More Generally}
In general, start with a $(t-2)$-regular graph (every vertex has degree $t-2$), and add an edge between two previously nonadjacent vertices. Assuming such a graph and such a pair exist, the average degree becomes $t-2+2/n$. The \ES\ theorem guarantees every $t$-vertex tree, even though only two vertices now have degree $t-1$.

At the other extreme, suppose the number of edges works out to $\binom{k}{2}+k(n-k)$: exactly enough to make $k$ vertices universal, with no edges among the others. For $n>k+1$, there are too few edges to make another vertex universal. The average degree is $2k-k(k+1)/n$, although the other vertices have degree only $k$. Take $k=\lfloor t/2\rfloor$ and $n$ sufficiently large; this average exceeds $t-2$, and every $t$-vertex tree is easy to find directly. Color its vertices with two colors so that every edge joins opposite colors. Put the smaller color class at universal vertices and the larger among the remaining vertices. There is room for both classes, and every required edge is present.

\Needspace{6\baselineskip}
\section{Revealing a Graph One Vertex at a Time}
Let's now discuss the proof of the Erd\H{o}s--S\'os conjecture. Fix a graph $G$ on $n\ge2$ vertices, and choose an ordering
\[
\pi=(v_0,v_1,\ldots,v_{n-1})
\]
of the vertices of $G$ in which each vertex appears exactly once. Let $\mathcal O$ denote the set of all $n!$ such orderings. Start with $v_0$, then reveal the other vertices one at a time. Whenever a vertex is revealed, also reveal all its edges to vertices already present. At each stage we therefore see an induced subgraph of $G$.

The subscripts record positions in this particular ordering, starting at zero; they are not permanent vertex labels. For example, if $\pi=(a,b,c,d,e)$, then $v_0=a$, while the first vertex of $(d,a,e,c,b)$ is $d$. The graph $G$ itself stays fixed. Revealing vertices changes only how much of it we inspect.

\begin{definition}[The revealed graph]\label{def:revealed}
For a set $S\subseteq V(G)$, let $G[S]$ denote the induced subgraph on $S$, keeping all edges of $G$ between its vertices. For $0\le j\le n-1$, write
\[
\Gcut{\pi}{j}=G[\{v_0,\ldots,v_j\}]
\]
for the graph revealed at the arrival of $v_j$.
\end{definition}
Each stage has exactly one associated graph. It includes the newly arrived vertex $v_j$, so it has $j+1$ vertices. Initially we have $\Gcut{\pi}{0}$ which consists of $v_0$ alone; at the end we have $\Gcut{\pi}{n-1}=G$.

For example, take the graph below and choose the ordering
$\pi=(b,e,a,d,c)$. Any ordering is allowed; consecutive vertices need not be adjacent.

\begin{pictureblock}

\hostpic{a/v,b/v,c/v,d/v,e/v}{a/b,b/c,a/d}{a/b,b/c,a/d}

\captionof{figure}{Our fixed host graph $G$, with edges $ab,bc,ad$
and isolated vertex $e$.}
\label{fig:host-reveal}

\end{pictureblock}

Starting with $b$ at stage zero, we reveal $e,a,d,c$ in that order.
The five stages are shown below. Notice that we retain
\emph{all} edges between the vertices revealed so far.

\begin{pictureblock}

\threepanel{\hostpic{b/v}{}{}}
{$\Gcut{\pi}{0}$: start with $b$.\\
Revealed vertex: $b$.}
\hfill
\threepanel{\hostpic{b/v,e/v}{}{}}
{$\Gcut{\pi}{1}$: after $e$ arrives.\\
Revealed vertices: $b,e$.}
\hfill
\threepanel{\hostpic{a/v,b/v,e/v}{a/b}{a/b}}
{$\Gcut{\pi}{2}$: after $a$ arrives.\\
Revealed vertices: $b,e,a$.}

\par\bigskip

\panel{\hostpic{a/v,b/v,d/v,e/v}{a/b,a/d}{a/b,a/d}}
{$\Gcut{\pi}{3}$: after $d$ arrives.\\
Revealed vertices: $b,e,a,d$.}
\hfill
\panel{\hostpic{a/v,b/v,c/v,d/v,e/v}{a/b,b/c,a/d}{a/b,b/c,a/d}}
{$\Gcut{\pi}{4}=G$: after $c$ arrives.\\
Every vertex has now been revealed.}

\captionof{figure}{The successive induced subgraphs, from the
single vertex at stage zero to the entire graph at stage four.
Vertices keep the same relative positions in the drawings so that the
changes are easy to follow; their arrival order is specified
by $\pi$.}
\label{fig:reveal-stages}

\end{pictureblock}

\Needspace{6\baselineskip}
\section{Rooted Trees and Early Neighbors}
We now use a classic strategy in combinatorics: make the statement we are trying to prove stronger, so that the added structure makes the proof easier. Instead of asking only for an unrooted copy of a tree, we specify a root and require it to land at the first vertex of an ordering. This gives us a designated place at which to attach an edge or join smaller trees. We will prove a counting bound for every choice of root of the tree; the first vertex of $G$ is allowed to vary with the ordering.

Fix a tree $T$ with a distinguished vertex $r$, called its \emph{root}, and write $t=|V(T)|\ge2$. For an ordering beginning at $x=v_0$, we say that $\Gcut{\pi}{j}$ \emph{contains $T$ rooted at $x$} if it contains a copy of $T$ in which $r$ is represented by $x$.

Our goal is to reveal a copy of $T$ rooted at the first vertex. We count the neighbors encountered before we achieve that goal.
\Needspace{18\baselineskip}
\begin{definition}[Early neighbors]\label{def:early}
Given an ordering $\pi=(v_0,\ldots,v_{n-1})$, start with $v_0$ and reveal the remaining vertices one at a time, retaining all edges between the vertices revealed so far. Stop as soon as a copy of $T$ rooted at $v_0$ appears.

The \emph{early neighbors} are the neighbors of $v_0$ that arrived \emph{strictly before} this stopping stage. In particular, if the vertex that completes the first rooted copy is itself a neighbor of $v_0$, it does not count. If no such copy ever appears, then all $\deg(v_0)$ neighbors of $v_0$ will be considered early.

Let $b_T(\pi)$ denote the number of early neighbors for the ordering $\pi$, and define
\[
D(T)=\sum_{\pi\in\mathcal O}b_T(\pi),
\]
where the sum runs over all $n!$ vertex orderings. The host graph $G$ is fixed throughout, so its dependence is suppressed in $b_T$ and $D(T)$.
\end{definition}

\subsection*{Four Examples: The Same Ordering, Different Rooted Trees}
Use the same graph $G$ and ordering $\pi=(b,e,a,d,c)$ as in
Figures~\ref{fig:host-reveal}--\ref{fig:reveal-stages}.
In each run, the target's root must be represented by $b$.
\begin{pictureblock}
\panel{\begin{tikzpicture}
\node[root](r)at(0,0){$r$};\node[v](s)at(1.4,0){$s$};
\draw[chosen](r)--(s);
\end{tikzpicture}}{$P_2$ rooted at an endpoint.}
\hfill
\panel{\begin{tikzpicture}
\node[root](r)at(0,0){$r$};\node[v](s)at(1.4,0){$s$};\node[v](z)at(2.8,0){$z$};
\draw[chosen](r)--(s)--(z);
\end{tikzpicture}}{$P_3$ rooted at an endpoint.}
\par\bigskip
\panel{\begin{tikzpicture}
\node[v](s)at(0,0){$s$};\node[root](r)at(1.2,0){$r$};\node[v](z)at(2.4,0){$z$};\node[v](w)at(3.6,0){$w$};
\draw[chosen](s)--(r)--(z)--(w);
\end{tikzpicture}}{$P_4$ rooted at an internal vertex.}
\hfill
\panel{\begin{tikzpicture}
\node[root](r)at(0,0){$r$};\node[v](s)at(1.2,0){$s$};\node[v](z)at(2.4,0){$z$};\node[v](w)at(3.6,0){$w$};
\draw[chosen](r)--(s)--(z)--(w);
\end{tikzpicture}}{$P_4$ rooted at an endpoint.}
\captionof{figure}{Four choices of rooted target. The thick blue outline and shaded interior identify the root $r$, which must be represented by $b$.}\label{fig:rooted-target}
\end{pictureblock}
In every sequence, vertices keep their positions from Figure~\ref{fig:reveal-stages}, and all edges between revealed vertices are retained. Blue shading marks the root $b$; orange marks the new arrival. Thick blue edges highlight a rooted copy only when the run succeeds.

\par\addvspace{10pt}\noindent\begin{minipage}{\linewidth}
\paragraph{$P_2$ Rooted at an Endpoint.}\mbox{}\par
\begin{pictureblock}
\begin{minipage}[t]{.165\linewidth}\centering
\hostpic{b/root}{}{}
\par\smallskip\footnotesize
\textbf{Stage 0}: $b$ alone.\\[2pt]
Start.
\end{minipage}%
\hfill $\rightarrow$ \hfill
\begin{minipage}[t]{.165\linewidth}\centering
\hostpic{b/root,e/arrival}{}{}
\par\smallskip\footnotesize
\textbf{Stage 1}: $e$ arrives.\\[2pt]
Not a neighbor.
\end{minipage}%
\hfill $\rightarrow$ \hfill
\begin{minipage}[t]{.165\linewidth}\centering
\hostpic{b/root,e/v,a/arrival}{a/b}{a/b}
\par\smallskip\footnotesize
\textbf{Stage 2}: $a$ arrives.\\[2pt]
\textbf{Stop: $T$ appears.}
\end{minipage}%

\captionof{figure}{Stop when $a$ reveals the rooted edge $b-a$. The first neighbor completes the tree, so it does not count as early; $d,c$ are still unrevealed.}\label{fig:stop-p2-endpoint}
\end{pictureblock}
\[
\text{In this example, } b_T(\pi)=0\qquad\text{(no neighbor is early).}
\]
\end{minipage}\par

\par\addvspace{10pt}\noindent\begin{minipage}{\linewidth}
\paragraph{$P_3$ Rooted at an Endpoint.}\mbox{}\par
\begin{pictureblock}
\begin{minipage}[t]{.165\linewidth}\centering
\hostpic{b/root}{}{}
\par\smallskip\footnotesize
\textbf{Stage 0}: $b$ alone.\\[2pt]
Start.
\end{minipage}%
\hfill $\rightarrow$ \hfill
\begin{minipage}[t]{.165\linewidth}\centering
\hostpic{b/root,e/arrival}{}{}
\par\smallskip\footnotesize
\textbf{Stage 1}: $e$ arrives.\\[2pt]
Not a neighbor.
\end{minipage}%
\hfill $\rightarrow$ \hfill
\begin{minipage}[t]{.165\linewidth}\centering
\hostpic{b/root,e/v,a/arrival}{a/b}{}
\par\smallskip\footnotesize
\textbf{Stage 2}: $a$ arrives.\\[2pt]
Count $a$.
\end{minipage}%
\hfill $\rightarrow$ \hfill
\begin{minipage}[t]{.165\linewidth}\centering
\hostpic{b/root,e/v,a/v,d/arrival}{a/b,a/d}{a/b,a/d}
\par\smallskip\footnotesize
\textbf{Stage 3}: $d$ arrives.\\[2pt]
\textbf{Stop: $T$ appears.}
\end{minipage}%

\captionof{figure}{Stop when $d$ reveals the rooted path $b-a-d$; $c$ is still unrevealed.}\label{fig:early-example}
\end{pictureblock}
\[
\text{In this example, } b_T(\pi)=1\qquad\text{(only $a$ is early).}
\]
\end{minipage}\par

\par\addvspace{10pt}\noindent\begin{minipage}{\linewidth}
\paragraph{$P_4$ Rooted at an Internal Vertex.}\mbox{}\par
\begin{pictureblock}
\begin{minipage}[t]{.165\linewidth}\centering
\hostpic{b/root}{}{}
\par\smallskip\footnotesize
\textbf{Stage 0}: $b$ alone.\\[2pt]
Start.
\end{minipage}%
\hfill $\rightarrow$ \hfill
\begin{minipage}[t]{.165\linewidth}\centering
\hostpic{b/root,e/arrival}{}{}
\par\smallskip\footnotesize
\textbf{Stage 1}: $e$ arrives.\\[2pt]
Not a neighbor.
\end{minipage}%
\hfill $\rightarrow$ \hfill
\begin{minipage}[t]{.165\linewidth}\centering
\hostpic{b/root,e/v,a/arrival}{a/b}{}
\par\smallskip\footnotesize
\textbf{Stage 2}: $a$ arrives.\\[2pt]
Count $a$.
\end{minipage}%
\hfill $\rightarrow$ \hfill
\begin{minipage}[t]{.165\linewidth}\centering
\hostpic{b/root,e/v,a/v,d/arrival}{a/b,a/d}{}
\par\smallskip\footnotesize
\textbf{Stage 3}: $d$ arrives.\\[2pt]
Not a neighbor.
\end{minipage}%
\hfill $\rightarrow$ \hfill
\begin{minipage}[t]{.165\linewidth}\centering
\hostpic{b/root,e/v,a/v,d/v,c/arrival}{a/b,b/c,a/d}{a/b,b/c,a/d}
\par\smallskip\footnotesize
\textbf{Stage 4}: $c$ arrives.\\[2pt]
\textbf{Stop: $T$ appears.}
\end{minipage}%

\captionof{figure}{Stop when $c$ reveals the rooted path $c-b-a-d$, with $b$ internal.}\label{fig:stop-p4-internal}
\end{pictureblock}
\[
\text{In this example, } b_T(\pi)=1\qquad\text{(only $a$ is early; $c$ completes the copy).}
\]
\end{minipage}\par

\par\addvspace{10pt}\noindent\begin{minipage}{\linewidth}
\paragraph{$P_4$ Rooted at an Endpoint.}\mbox{}\par
\begin{pictureblock}
\begin{minipage}[t]{.165\linewidth}\centering
\hostpic{b/root}{}{}
\par\smallskip\footnotesize
\textbf{Stage 0}: $b$ alone.\\[2pt]
Start.
\end{minipage}%
\hfill $\rightarrow$ \hfill
\begin{minipage}[t]{.165\linewidth}\centering
\hostpic{b/root,e/arrival}{}{}
\par\smallskip\footnotesize
\textbf{Stage 1}: $e$ arrives.\\[2pt]
Not a neighbor.
\end{minipage}%
\hfill $\rightarrow$ \hfill
\begin{minipage}[t]{.165\linewidth}\centering
\hostpic{b/root,e/v,a/arrival}{a/b}{}
\par\smallskip\footnotesize
\textbf{Stage 2}: $a$ arrives.\\[2pt]
Count $a$.
\end{minipage}%
\hfill $\rightarrow$ \hfill
\begin{minipage}[t]{.165\linewidth}\centering
\hostpic{b/root,e/v,a/v,d/arrival}{a/b,a/d}{}
\par\smallskip\footnotesize
\textbf{Stage 3}: $d$ arrives.\\[2pt]
Not a neighbor.
\end{minipage}%
\hfill $\rightarrow$ \hfill
\begin{minipage}[t]{.165\linewidth}\centering
\hostpic{b/root,e/v,a/v,d/v,c/arrival}{a/b,b/c,a/d}{}
\par\smallskip\footnotesize
\textbf{Stage 4}: $c$ arrives.\\[2pt]
Count $c$; no copy.
\end{minipage}%


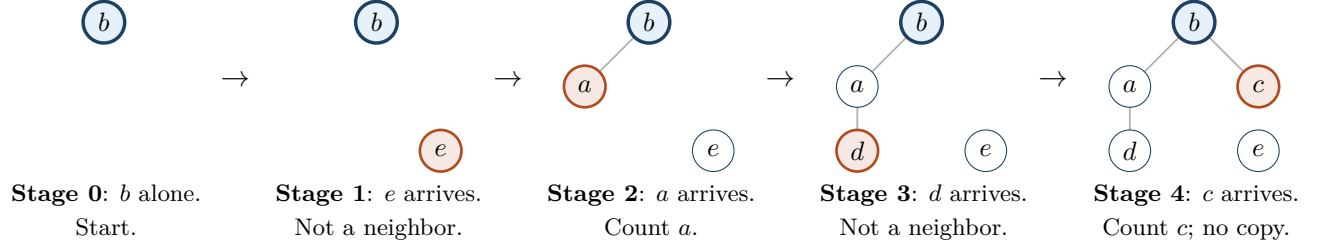
\captionof{figure}{Exhaust the ordering: all of $G$ is revealed, but no $P_4$ has an endpoint at $b$.}\label{fig:exhaust-p4-endpoint}
\end{pictureblock}
\[
\text{In this example, } b_T(\pi)=2=\deg_G(b)\qquad\text{(both $a$ and $c$ are early).}
\]
\end{minipage}\par

\Needspace{6\baselineskip}
For use in the proof, we can also recognize an early neighbor directly from the revealed graph. A neighbor $y=v_j$ of $x=v_0$ is early exactly when $\Gcut{\pi}{j}$ does not contain $T$ rooted at $x$. This graph includes $y$ itself. Once a rooted copy appears, it remains present in every larger revealed graph.

In the proof, an \emph{early-neighbor occurrence} $(\pi,y)$ records a full ordering $\pi$ and one neighbor $y$ counted by $b_T(\pi)$. Thus $D(T)$ counts these occurrences, not copies of the tree. Different full orderings count separately, even if they reveal the same vertices before stopping. We retain the unrevealed tail so that the rearrangements in the proof can be undone. For any rooted target $U$, write $\mathcal E(U)$ for its set of early-neighbor occurrences; thus $|\mathcal E(U)|=D(U)$. In an occurrence $(\pi,y)$, we call $y$ the \emph{chosen early neighbor}, always specifying the target when needed.

\Needspace{6\baselineskip}
\section{Induction on the Rooted Tree}
The quantity $D(T)$ counts how many neighbor arrivals, in total over all orderings, occur before a rooted copy appears. Our aim is to bound this total using only $t$ and $n$.
\Needspace{10\baselineskip}
\begin{proposition}\label{prop:count}
Let $G$ be a graph with $n\ge2$ vertices, and let $T$ be a rooted tree on $t\ge2$ vertices. Then
\begin{equation}
D(T)=\sum_{\pi}b_T(\pi)\le(t-2)n!.
\label{eq:main}
\end{equation}
Equivalently, the stopping procedure counts at most $t-2$ early neighbors on average over all orderings, including those in which no rooted copy ever appears.
\end{proposition}
\begin{proof}
We induct on $t$, proving the assertion for every rooted tree at each size, including every choice of root.

For the base case, let $t=2$ and $\pi$ be any ordering of $G$'s vertices. Since $t=2$, the tree is a single edge. If $v_0$ has at least one neighbor, then the first neighbor of $v_0$ to arrive immediately completes this rooted tree, so we stop without counting that neighbor. Thus, we have $b_T(\pi) = 0$. Otherwise, $v_0$ is an isolated vertex and $G$ contains no copy of this rooted edge at $v_0$, so in such a circumstance we will exhaust the ordering with no neighbors to count. And so, again, we have $b_T(\pi)=0$. Therefore, $D(T)=0$, which satisfies $D(T)=\sum_{\pi}b_T(\pi)\le(t-2)n!$, as desired.

Now let $t\ge3$, and assume the proposition holds for all smaller rooted trees with at least two vertices. Because trees are connected, the root $r$ of $T$ has positive degree. There are two cases to consider.

\subsection*{Case 1: The Root Is a Leaf}
Let $s$ be the unique neighbor of $r$. Remove $r$ to obtain a tree $T'$ on $t-1$ vertices, rooted at $s$.  By the inductive hypothesis, $D(T') \leq (t-3)n!$.
\begin{pictureblock}
\panel{\begin{tikzpicture}
\node[root](r)at(0,0){$r$};
\node[v](s)at(1.2,0){$s$};
\node[v](u)at(2.4,0){$u$};
\node[v](w)at(3.6,0){$w$};
\draw[accent,thick](r)--(s);
\draw[chosen](s)--(u)--(w);
\end{tikzpicture}}{$T$: remove the leaf root and its edge.}
\hfill
\panel{\begin{tikzpicture}
\node[root](s)at(0,0){$s$};
\node[v](u)at(1.2,0){$u$};
\node[v](w)at(2.4,0){$w$};
\draw[chosen](s)--(u)--(w);
\end{tikzpicture}}{$T'$: its root is the former neighbor $s$.}
\captionof{figure}{For example, deleting the endpoint root of $P_4$
gives $P_3$ rooted at an endpoint. The roots of $T$ and $T'$ are
different tree vertices.}\label{fig:leaf-trees}
\end{pictureblock}

Both $D(T)$ and $D(T')$ count early-neighbor occurrences with respect to orderings of the \emph{vertices} of $G$. Our goal is to prove
\[
D(T)\le D(T')+n!.
\]
We will compare the two collections of occurrences, but first let us explain why the comparison involves looking for $T'$ at a neighbor of the first vertex.

\paragraph{Which Vertices Play the Roles of $r$ and $s$?}
Take an early-neighbor occurrence $(\pi,y)$ for $T$. Let $x$ be the first vertex of $\pi$, let $j$ be the position of $y$, and put
\[
H=\Gcut{\pi}{j}.
\]
Thus $xy$ is an edge, and $H$ contains no $T$ rooted at $x$.

The vertices $r,s$ belong to the abstract target tree, whereas $x,y$ belong to $G$. In a copy of $T$ rooted at $x$, the vertex $x$ would represent $r$. Since $r$ has just one neighbor, $s$, we can try using the known edge $xy$ to represent $rs$. This makes $y$ the natural candidate to represent $s$, the root of $T'$.

That is why we ask whether $H$ contains $T'$ \emph{rooted at $y$}. We are choosing where to look, not asserting that a copy already exists. Also, it is $r$ that is a leaf of $T$; we do not assume that $x$ has degree one in $G$. A copy may leave other edges at $x$ unused.

\paragraph{Why Finding $T'$ Need Not Already Give $T$.}
Since we are assuming that $H$ does not have a $T$ rooted at $x$, one may think that it is impossible to have a $T'$ rooted at $y$, because if it did, couldn't we simply add $xy$ to such a copy and obtain the forbidden $T$ rooted at $x$?

We could, \emph{provided that the chosen copy of $T'$ does not already use $x$}. Then $x$ is a new vertex, and attaching it at $y$ makes it the required leaf root. But if the copy already uses $x$, we cannot use $x$ again as that additional leaf. The essential obstruction is the repeated vertex, whether or not the edge $xy$ is already part of the copy.

Consequently, because $H$ contains no $T$ rooted at $x$,
\[
\boxed{\text{Every copy of $T'$ rooted at $y$ in $H$, if any, must use $x$.}}
\]
There may be no such copy at all, or there may be copies that all use $x$. Both situations are possible, and so we will partition the early neighbor occurrences $(\pi,y)$ for $T$ into two sets based on whether or not $H$ has a $T'$ rooted at $y$.

\paragraph{Two Examples Before We Partition the Count.}
Use the paths from Figure~\ref{fig:leaf-trees}: $T=P_4$ and $T'=P_3$, both rooted at endpoints. Take our usual host graph with edges $ab,bc,ad$ and isolated vertex $e$, but use the ordering
\[
\pi=(a,b,c,d,e),\qquad x=a.
\]

First, let's look at the stage when the early neighbor $b$ arrives. Here $H$ is just the edge $ab$. It contains neither $T$ rooted at $a$ nor $T'$ rooted at $b$. Thus $b$ is early for $T$, and the smaller rooted tree is absent too.

Next, in the same ordering, let's look at the stage when the early neighbor $d$ arrives. Now $H$ is the path $d-a-b-c$. It contains no endpoint-rooted $P_4$ at $a$, since $a$ is internal in that path. But it does contain the endpoint-rooted $P_3$ at $d$ given by $d-a-b$. This copy already uses $x=a$, so we cannot attach $a$ as a new leaf. Thus $d$ is early for $T$ even though $T'$ rooted at $d$ is present.
\begin{pictureblock}
\panel{\hostpic{a/root,b/arrival}{a/b}{}}
{$y=b$: $H=G[\{a,b\}]$.\\No $T'$ is rooted at $b$.}
\hfill
\panel{\hostpic{a/root,b/v,c/v,d/arrival}{a/b,b/c,a/d}{a/d,a/b}}
{$y=d$: $H=G[\{a,b,c,d\}]$.\\The rooted $T'$ is $d-a-b$; it uses $a$.}
\captionof{figure}{Two early-neighbor occurrences for $T$. Blue shading identifies the first vertex $x=a$, and orange identifies the chosen early neighbor $y$. In the right panel the highlighted copy of $T'$ is rooted at $d$, not at $a$. All vertices keep their positions from the host drawing.}\label{fig:leaf-preview}
\end{pictureblock}

\paragraph{Divide Both Counts into Two Parts.}
We now give names to the two situations just illustrated and make the corresponding division for $T'$.

On the $D(T)$ side, use the notation above: $(\pi,y)$ is an early occurrence, $x$ is first, $y$ is in position $j$, and $H=\Gcut{\pi}{j}$.

Independently, on the $D(T')$ side, consider an early occurrence $(\sigma,x)$. Here $y$ is the first vertex, $x$ is its chosen early neighbor in some position $k$, and
\[
K=\Gcut{\sigma}{k}.
\]
The letters on this side anticipate the exchange of roles we will use below. At this point $\pi$ and $\sigma$ are arbitrary orderings on their respective sides; we have not yet paired them, and we do not yet assume $H=K$.

\begin{pictureblock}
\small
\begin{tabular}{@{}p{.47\linewidth}p{.47\linewidth}@{}}
\toprule
Occurrences $(\pi,y)$ counted by $D(T)$
& Occurrences $(\sigma,x)$ counted by $D(T')$\\
\midrule
First vertex: $x$; chosen early neighbor: $y$.\newline
$xy\in E(G)$ and $H=\Gcut{\pi}{j}$.\newline
Always: $H$ contains no $T$ rooted at $x$.
&
First vertex: $y$; chosen early neighbor: $x$.\newline
$xy\in E(G)$ and $K=\Gcut{\sigma}{k}$.\newline
Always: $K$ contains no $T'$ rooted at $y$.\\
\midrule
$A$: $H$ contains no $T'$ rooted at $y$.
& $A'$: $K$ contains no $T$ rooted at $x$.\\[6pt]
$B$: $H$ contains $T'$ rooted at $y$.
& $B'$: $K$ contains $T$ rooted at $x$.\\
\bottomrule
\end{tabular}
\captionof{table}{Two independently defined partitions. On the left we seek $T$ at $x$ and additionally test for $T'$ at $y$; on the right we seek $T'$ at $y$ and additionally test for $T$ at $x$.}\label{tab:leaf-partitions}
\end{pictureblock}

Our first example, $(\pi,b)$, belongs to $A$; our second, $(\pi,d)$, belongs to $B$. In general,
\[
D(T)=|A|+|B|,
\qquad
D(T')=|A'|+|B'|.
\]
Here is the plan: construct a \emph{bijection between $A$ and $A'$}, and prove that $|B|\le n!$. The unused occurrences in $B'$ only make $D(T')$ larger.

\paragraph{The Swap Operation.}
For an occurrence $(\pi,y)$ with first vertex $x$, write $\pi=(x,L,y,R)$, where $L$ lists the vertices between $x$ and $y$ and $R$ lists those after $y$. Define the \emph{swapped ordering}
\begin{equation}
\pi=(x,L,y,R)\quad\longmapsto\quad
\sigma=\pi^*=(y,L,x,R).
\label{eq:leafmap}
\end{equation}
Only $x$ and $y$ exchange positions; the lists $L,R$ are unchanged and either may be empty. The notation $\pi^*$ refers to this particular choice of $y$: different early neighbors in the same ordering give different swaps. Swapping $x,y$ again recovers $\pi$.

\paragraph{Match $A$ with $A'$.}
At the arrival of $y$ in $\pi$ and of $x$ in $\sigma$, we have revealed exactly the same vertices, namely $x,y$ and those in $L$. Thus the induced graphs $H,K$ in Table~\ref{tab:leaf-partitions} are equal. Membership in either $A$ or $A'$ asks for exactly the same two absences in that graph: no $T$ rooted at $x$ and no $T'$ rooted at $y$. Since $xy$ remains an edge, the swap therefore pairs
\[
(\pi,y)\in A\quad\longleftrightarrow\quad(\sigma,x)\in A'.
\]
Swapping back reverses the pairing, so it is a bijection and $|A|=|A'|$.

\emph{Return to the first example.}
Swapping $a$ with the chosen early neighbor $b$ changes $(a,b,c,d,e)$ to $(b,a,c,d,e)$. The revealed graph is still just the edge $ab$; the first vertex and chosen early neighbor exchange roles.
\begin{pictureblock}
\panel{\hostpic{a/root,b/arrival}{a/b}{}}
{$((a,b,c,d,e),b)\in A$.\\First vertex $a$; $b$ is early for $T$.}
\hfill
\panel{\hostpic{a/arrival,b/root}{a/b}{}}
{$((b,a,c,d,e),a)\in A'$.\\First vertex $b$; $a$ is early for $T'$.}
\captionof{figure}{The first example after applying the swap. The graph and vertex positions stay fixed; the roles, shown by the shading, change.}\label{fig:leaf-early}
\end{pictureblock}

\Needspace{6\baselineskip}
\paragraph{Bound the Remaining Group $B$.}
Now take $(\pi,y)\in B$, again writing $\pi=(x,L,y,R)$ and $H=\Gcut{\pi}{j}$. The graph $H$ contains no $T$ rooted at $x$, but does contain $T'$ rooted at $y$. As we established before the partition, every such copy uses $x$. Thus $H-x$, the induced graph obtained by removing $x$, contains no $T'$ rooted at $y$.

Apply the same swap, giving $\sigma=\pi^*=(y,L,x,R)$. Immediately before $x$ arrives, the revealed graph is exactly $H-x$; when $x$ arrives, it becomes $H$. Therefore $x$ is the vertex whose arrival first produces $T'$ rooted at $y$.

Since $xy$ is an edge, we can call $x$ the \emph{stopping neighbor} for $T'$ in $\sigma$. It is not early for $T'$, so $(\sigma,x)$ belongs to neither $A'$ nor $B'$. We will instead count the swapped orderings themselves.

\emph{Return to the second example.}
Swapping $x=a$ with $y=d$ gives
\[
(a,b,c,d,e)\quad\longmapsto\quad(d,b,c,a,e).
\]
The intervening list $(b,c)$ keeps its order. Before $a$ arrives in the new ordering, the edge $bc$ is present but the first vertex $d$ is isolated. When $a$ arrives, the rooted path $d-a-b$ appears and the run for $T'$ stops.
\begin{pictureblock}
\panel{\hostpic{b/v,c/v,d/root}{b/c}{}}
{Before $a$ arrives in $(d,b,c,a,e)$.\\The root $d$ is isolated.}
\hfill
\panel{\hostpic{a/arrival,b/v,c/v,d/root}{a/b,b/c,a/d}{a/d,a/b}}
{After $a$ arrives.\\The first rooted $T'$ appears; $a$ is not early.}
\captionof{figure}{The second example after applying the swap. The selected $T'$ is rooted at $d$ and uses the newly arrived vertex $a$. The gray edge $bc$ is retained but unused by this copy.}\label{fig:leaf-stop}
\end{pictureblock}

\emph{Count the swapped orderings.}
Let $B^*$ be the set of swapped orderings obtained from occurrences in $B$. These form a subset of $\mathcal O$, the set of all vertex orderings: in each one, the first copy of $T'$ rooted at the first vertex appears at a neighbor of that vertex.

Different occurrences in $B$ give different members of $B^*$. Indeed, from a swapped ordering $\sigma$ we know its first vertex $y$, and we recover $x$ as the unique vertex whose arrival first produces $T'$ rooted at $y$. Swapping $x,y$ back recovers the entire original occurrence $(\pi,y)$. Hence
\[
|B|=|B^*|\le|\mathcal O|=n!.
\]
The key is the \emph{first} appearance of $T'$: a vertex arriving later would not uniquely identify the swap. This count concerns the resulting orderings, not the original orderings.

\paragraph{Compare the Two Totals.}
Since $A'$ is a subset of the occurrences counted by $D(T')$, our two comparisons give
\begin{equation}
D(T)=|A|+|B|=|A'|+|B|
\le D(T')+n!.
\label{eq:leafD}
\end{equation}
By the inductive hypothesis, $D(T')\le(t-3)n!$, so
\[
D(T)\le(t-3)n!+n!=(t-2)n!.
\]
This completes the leaf case.

\Needspace{6\baselineskip}
\subsection*{Case 2: The Root Has Degree at Least Two}
Remove $r$ temporarily. Partition the remaining components into two nonempty groups, and adjoin $r$ to each group, together with its edges to the vertices in that group. This produces rooted trees $T_1,T_2$, both rooted at $r$, such that
\[
T_1\cup T_2=T,\qquad V(T_1)\cap V(T_2)=\{r\}.
\]
Writing $t_i=|V(T_i)|$, we have $2\le t_i<t$ and $t_1+t_2=t+1$, since the root is counted twice.

Here $r$ is a vertex of the target trees, while $x$ will denote the first vertex of an ordering of $G$. We seek both smaller trees rooted at $x$, so $x$ represents their common root $r$. Unlike Case 1, the rearrangement in this case will keep $x$ first.
\begin{pictureblock}
\threepanel{\begin{tikzpicture}
\path[use as bounding box] (-1.9,-.4) rectangle (1.9,1.1);
\node[root](r)at(0,0){$r$};\node[v](a)at(-.8,.7){$s_1$};\node[v](b)at(-1.6,.7){$u_1$};\node[v](c)at(.8,.7){$s_2$};\node[v](d)at(1.6,.7){$u_2$};
\draw[chosen](r)--(a)--(b);\draw[accent,thick](r)--(c)--(d);
\end{tikzpicture}}{$T$: both groups of branches.}
\hfill
\threepanel{\begin{tikzpicture}
\path[use as bounding box] (-1.9,-.4) rectangle (1.9,1.1);
\node[root](r)at(0,0){$r$};\node[v](a)at(-.8,.7){$s_1$};\node[v](b)at(-1.6,.7){$u_1$};\draw[chosen](r)--(a)--(b);\end{tikzpicture}}{$T_1$: one group and the root.}
\hfill
\threepanel{\begin{tikzpicture}
\path[use as bounding box] (-1.9,-.4) rectangle (1.9,1.1);
\node[root](r)at(0,0){$r$};\node[v](c)at(.8,.7){$s_2$};\node[v](d)at(1.6,.7){$u_2$};\draw[accent,thick](r)--(c)--(d);\end{tikzpicture}}{$T_2$: the other group and root.}
\captionof{figure}{Splitting $T$ into two smaller rooted trees. Copies of $T_1,T_2$ combine into $T$ if they share only their root. Merely finding both copies is insufficient if other vertices overlap.}\label{fig:branch-trees}
\end{pictureblock}

\paragraph{Divide the Count into Three Parts.}
Consider an early-neighbor occurrence $(\pi,y)$ for $T$, with first vertex $x$ and chosen early neighbor $y=v_j$. Put $H=\Gcut{\pi}{j}$. Thus $xy\in E(G)$, and $H$ contains no $T$ rooted at $x$. Partition all these occurrences into three sets according to when the smaller tree $T_1$ first appears, always rooted at $x$:
\begin{pictureblock}
\small
\begin{tabular}{@{}c p{.59\linewidth}p{.25\linewidth}@{}}
\toprule
Set & Membership condition & Comparison to prove\\
\midrule
$S_1$ & $H$ contains no $T_1$ rooted at $x$. & $|S_1|=D(T_1)$\\[5pt]
$S_2$ & $T_1$ rooted at $x$ first appeared at a stage $k<j$. & $|S_2|\le D(T_2)$\\[5pt]
$S_0$ & $T_1$ rooted at $x$ first appears at stage $j$, when $y$ arrives. & $|S_0|\le n!$\\
\bottomrule
\end{tabular}
\captionof{table}{The three parts of the count in Case 2. Every member of every set is an early-neighbor occurrence for $T$; the smaller tree $T_1$ determines which part it belongs to.}\label{tab:branch-partitions}
\end{pictureblock}
These sets are disjoint and exhaust the occurrences counted by $D(T)$, so
\[
D(T)=|S_1|+|S_2|+|S_0|.
\]
Here is the plan. The set $S_1$ will turn out to be exactly the set of early-neighbor occurrences for $T_1$. We will pair each member of $S_2$ with a distinct early-neighbor occurrence for $T_2$; some occurrences for $T_2$ may remain unused. Finally, at most one member of $S_0$ can come from each original ordering. Adding these three comparisons will give the desired bound.

For the examples in this case, take the $T=P_5$ pictured in Figure~\ref{fig:branch-trees}, rooted at its middle vertex, with $T_1$ and $T_2$ both copies of $P_3$ rooted at an endpoint. Use our familiar host graph with edges $ab,bc,ad$ and isolated vertex $e$.

\paragraph{Identify $S_1$ with All Early Occurrences for $T_1$.}
Membership in $S_1$ asks for two absences in $H$: no $T$ rooted at $x$ and no $T_1$ rooted at $x$. But the second already implies the first, since every rooted copy of $T$ contains a rooted copy of $T_1$. Thus the extra requirement adds nothing:
\[
S_1=\mathcal E(T_1),
\qquad\text{so}\qquad |S_1|=D(T_1).
\]
These are exactly the same ordering--neighbor pairs; no rearrangement is needed.

\emph{Example.} Choose $\pi=(a,e,b,d,c)$ and the early neighbor $y=d$. At its arrival, the revealed graph contains the path $b-a-d$ and the isolated vertex $e$. A three-vertex path is present, but $a$ is its middle vertex, not an endpoint. Thus $T_1=P_3$ rooted at an endpoint is still absent at $a$, and so is $T=P_5$ rooted at its middle. The same occurrence $((a,e,b,d,c),d)$ belongs to both $S_1$ and $\mathcal E(T_1)$.
\begin{pictureblock}
\hostpic{a/root,b/v,d/arrival,e/v}{a/b,a/d}{}

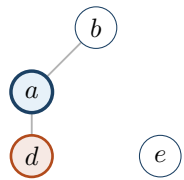
\captionof{figure}{When $d$ arrives in $(a,e,b,d,c)$, the path $b-a-d$ has the wrong root for $T_1$. Therefore $d$ is early for both $T_1$ and $T$. The vertex $c$ is still unrevealed.}\label{fig:branch-early}
\end{pictureblock}

\Needspace{6\baselineskip}
\paragraph{Match $S_2$ with Some Early Occurrences for $T_2$.}
For an occurrence in $S_2$, write $y=v_j$ and let $k$ be the first stage at which $T_1$ appeared rooted at $x$. Then $1\le k<j$. We check every arrival when choosing $k$, including arrivals of nonneighbors of $x$.

Divide the ordering into blocks:
\[
\pi=\bigl(x,\underbrace{v_1,\ldots,v_k}_{R},
\underbrace{v_{k+1},\ldots,v_j}_{X},
\underbrace{v_{j+1},\ldots,v_{n-1}}_{Y}\bigr).
\]
The block $R$ ends at the first appearance of $T_1$, and the next block $X$ ends at the chosen early neighbor $y$. Both are nonempty; the trailing list $Y$ may be empty. Write $V(R)$ and $V(X)$ for the vertex sets of these lists.

The graph $G[\{x\}\cup V(R)]$ contains $T_1$ rooted at $x$. The graph $G[\{x\}\cup V(X)]$ cannot contain $T_2$ rooted there. If it did, these copies of $T_1$ and $T_2$ would share only $x$, and so combining their selected edges would give a rooted copy of $T$ in $\Gcut{\pi}{j}$. This would contradict $y$ being early for $T$.

We can therefore obtain an early-neighbor occurrence for $T_2$ by moving $R$ behind $y$. Swap the blocks $R,X$, preserving the order within each:
\begin{equation}
\pi=(x,R,X,Y)\quad\longmapsto\quad
\sigma=(x,X,R,Y).
\label{eq:splitmap}
\end{equation}
Keep the same chosen neighbor $y$. In $\sigma$, the graph revealed at its arrival is exactly
\[
\Gcut{\sigma}{|X|}=G[\{x\}\cup V(X)].
\]
We just proved that this graph does not contain $T_2$ rooted at $x$. Also, $xy$ is still an edge. Thus the run for $T_2$ has not stopped, and $y$ is early for $T_2$ in $\sigma$.

\emph{Example.} Choose $\pi=(a,b,c,e,d)$ and the early neighbor $y=d$. It is early for $T=P_5$, since the isolated vertex $e$ prevents the graph from containing any five-vertex tree. The smaller tree $T_1$ first appears when $c$ arrives, using the path $a-b-c$. This is strictly before $d$ arrives. Notice that $c$ is not a neighbor of $a$. We have
\[
x=a,\qquad R=(b,c),\qquad X=(e,d),\qquad Y=\varnothing,
\]
so the block swap gives
\[
\pi=(a,\underbrace{b,c}_{R},\underbrace{e,d}_{X})
\quad\longmapsto\quad
\sigma=(a,\underbrace{e,d}_{X},\underbrace{b,c}_{R}).
\]
At $d$'s arrival in the new ordering $\sigma$, only $a,e,d$ have been revealed. Their induced graph has just the edge $ad$, so $d$ is early for $T_2=P_3$ rooted at an endpoint.
\begin{pictureblock}
\panel{\hostpic{a/root,b/v,c/v,d/arrival,e/v}{a/b,b/c,a/d}{a/b,b/c}}
{Before the block swap: $\Gcut{\pi}{4}$.\\$T_1$ has appeared, but $T$ is still absent.}
\hfill
\panel{\hostpic{a/root,d/arrival,e/v}{a/d}{}}
{After the block swap: $\Gcut{\sigma}{2}$.\\$T_2$ is still absent; $b,c$ are unrevealed.}

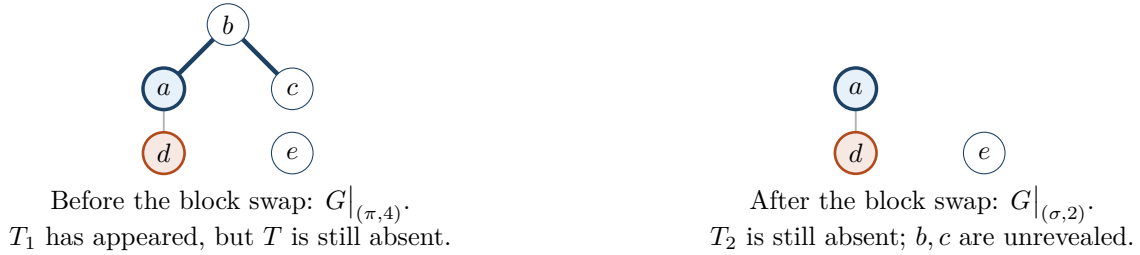
\captionof{figure}{Moving $R=(b,c)$ after the chosen early neighbor $d$ changes the graph revealed at its arrival, not the positions of the vertices in the drawing. The edge $ad$ remains, so $d$ is still a neighbor of $a$. All edges between revealed vertices are shown.}\label{fig:branch-move}
\end{pictureblock}

\emph{Count the rearranged occurrences.}
Let $S_2^*$ be the set of pairs $(\sigma,y)$ obtained by these block swaps. Keep the chosen early neighbor in each pair, since $D(T_2)$ counts occurrences, not just orderings. We have shown that
\[
S_2^*\subseteq\mathcal E(T_2).
\]
\Needspace{10\baselineskip}
There is also a useful way to recognize which members of $\mathcal E(T_2)$ lie in this subset. Given $(\sigma,y)$, let $x$ be first and let $X$ be the list after $x$ through $y$. Starting again with $x$, read the vertices \emph{after $y$}, ignoring all of $X$. Require that:
\begin{enumerate}[label=(\roman*),leftmargin=2em]
\item this separate reveal process reaches a first copy of $T_1$ rooted at $x$; call the list read up to that point $R$;
\item the combined graph $G[\{x\}\cup V(R)\cup V(X)]$ still contains no $T$ rooted at $x$.
\end{enumerate}
Both conditions hold for every pair produced by our block swap. Conversely, they let us uniquely reverse the swap: $(\sigma,y)$ determines $X$, the first appearance of $T_1$ in the separate search determines $R$, and the remainder is $Y$. Moving $R$ before $X$ recovers
\[
(\pi,y)=((x,R,X,Y),y).
\]
This occurrence belongs to $S_2$: $T_1$ appears at the end of $R$, strictly before $y$, while $T$ is still absent through $y$. Thus the two conditions describe exactly $S_2^*$, and every pair in it has exactly one original occurrence.
In the example, $\sigma=(a,e,d,b,c)$ and $y=d$ give $X=(e,d)$. Ignoring these vertices, the search from $a$ finds $T_1$ when $c$ arrives, recovering $R=(b,c)$. Moving this block back gives $\pi=(a,b,c,e,d)$.

The block swap is therefore a bijection from $S_2$ to the subset $S_2^*$ of $\mathcal E(T_2)$. Consequently,
\[
|S_2|=|S_2^*|\le|\mathcal E(T_2)|=D(T_2).
\]
\Needspace{6\baselineskip}
\paragraph{Bound the Remaining Group $S_0$.}
For an occurrence in $S_0$, the arrival of $y$ stops the run for the smaller target $T_1$, although $y$ is still early for $T$.

\emph{Example.} Use the same graph and rooted trees, but choose $\pi=(a,c,b,e,d)$ and the early neighbor $y=b$. Before $b$ arrives, only $a,c$ have been revealed, and they are not adjacent. The arrival of $b$ creates the rooted path $a-b-c$, which is the first copy of $T_1$ rooted at $a$. There are still only three revealed vertices, so the five-vertex tree $T$ is absent and $b$ is early for $T$.
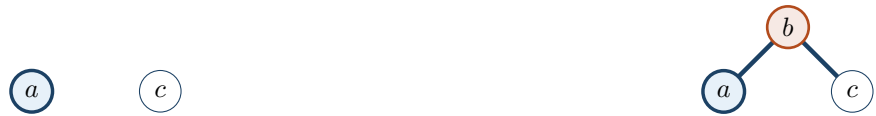
\begin{pictureblock}
\panel{\hostpic{a/root,c/v}{}{}}
{Before $b$ arrives in $(a,c,b,e,d)$.\\The rooted $T_1$ is absent.}
\hfill
\panel{\hostpic{a/root,b/arrival,c/v}{a/b,b/c}{a/b,b/c}}
{After $b$ arrives.\\The run for $T_1$ stops; the run for $T$ continues.}
\captionof{figure}{The arrival of the chosen early neighbor for $T$ first produces the smaller tree $T_1$ rooted at $a$. It does not produce $T$, so this occurrence is still counted in $D(T)$.}\label{fig:branch-stop}
\end{pictureblock}

\emph{Count the corresponding orderings.}
Let $S_0^*$ be the set of full vertex orderings with the following property: the first copy of $T_1$ rooted at the first vertex $x$ appears at a neighbor $y$ of $x$, and the graph revealed at that stage still contains no $T$ rooted at $x$.

Forgetting the chosen neighbor sends each $(\pi,y)\in S_0$ to an ordering $\pi\in S_0^*$. Conversely, an ordering in $S_0^*$ uniquely identifies $y$ as its stopping neighbor for $T_1$, recovering the occurrence $(\pi,y)\in S_0$. This is a bijection. Since $S_0^*$ is a subset of the set $\mathcal O$ of all orderings,
\[
|S_0|=|S_0^*|\le|\mathcal O|=n!.
\]
Unlike the bound for $B$, this bound uses the original ordering: no swap is needed.

\Needspace{12\baselineskip}
\paragraph{Add the Three Counts.}
We have now established all three comparisons in Table~\ref{tab:branch-partitions}. Adding the contributions from the three disjoint sets gives
\begin{equation}
D(T)=|S_1|+|S_2|+|S_0|\le D(T_1)+D(T_2)+n!.
\label{eq:splitD}
\end{equation}
Apply the inductive hypothesis to $T_1,T_2$ and use $t_1+t_2=t+1$:
\[
D(T)\le\bigl((t_1-2)+(t_2-2)+1\bigr)n!=(t-2)n!.
\]
This completes the second case.

The two cases exhaust the possible degrees of the root, completing the induction.
\end{proof}

\Needspace{24\baselineskip}
\section{Returning to the \ES\ Conjecture}
We can now see why this count answers the original question. First, recall the statement we set out to prove.
\begin{ESrestatement}
Let $n\ge t\ge2$. If $G$ is a finite simple graph on $n$ vertices with $\dd(G)>t-2$ (equivalently, $e(G)>(t-2)n/2$), then $G$ contains every tree on $t$ vertices.
\end{ESrestatement}
\begin{proof}[Proof of Theorem~\ref{thm:ES}]
Fix any tree $T$ on $t$ vertices and choose a root. Suppose, for a contradiction, that $G$ contains no copy of its underlying unrooted tree.

Then every run exhausts the ordering without finding a rooted copy of $T$. Every neighbor of the first vertex is therefore early, so $b_T(\pi)=\deg(v_0)$ for every ordering. Since exactly $(n-1)!$ orderings begin at each vertex $x$, the handshaking identity gives
\begin{equation}
D(T)=\sum_{\pi}\deg(v_0)
=(n-1)!\sum_{x\in V(G)}\deg(x)
=(n-1)! \cdot 2e(G).
\label{eq:total}
\end{equation}
The hypothesis $e(G)>(t-2)n/2$ now implies $D(T)>(t-2)n!$, contradicting Proposition~\ref{prop:count}. Thus $G$ contains $T$. Since $T$ was arbitrary, it contains every tree on $t$ vertices.
\end{proof}
The graph's degrees may be nearly equal or extremely unequal. The proof uses their total, not their distribution. The two maps move between different orderings, which is why one ordering may have many early neighbors without violating the bound on their total.

\Needspace{6\baselineskip}
\section{Proving Theorem~\ref{thm:ES} Directly and Probabilistically}\label{sec:alternative-endings}

The same counting bound gives two other ways to finish the proof of Theorem~\ref{thm:ES}. Throughout this section, keep its hypotheses $n\ge t\ge2$ and $\dd(G)>t-2$.

\subsection*{A Direct Proof}
Fix any tree $T$ on $t$ vertices and choose a root. By
Proposition~\ref{prop:count},
\[
\sum_{\pi} b_T(\pi)=D(T)\le (t-2)n!.
\]
On the other hand, each vertex is first in exactly $(n-1)!$
orderings, so the handshaking identity gives
\[
\sum_{\pi}\deg_G(v_0)
=(n-1)!\sum_{x\in V(G)}\deg_G(x)
=n!\,\overline d(G)
>(t-2)n!.
\]
Consequently,
\[
\sum_{\pi} b_T(\pi)<\sum_{\pi}\deg_G(v_0).
\]
There must therefore be an ordering $\pi$ for which
\[
b_T(\pi)<\deg_G(v_0).
\]
In this ordering, at least one neighbor $y=v_j$ of $v_0$ is not
early. Thus $\Gcut{\pi}{j}$ contains a copy of $T$ rooted at $v_0$,
and hence $G$ contains $T$. Since $T$ was arbitrary, $G$ contains
every tree on $t$ vertices. This completes the direct proof.

\subsection*{A Probabilistic Proof}
Probability offers another succinct way to organize the same argument. Fix a rooted tree $T$ on $t$ vertices. Choose an ordering $\Pi$ uniformly at random from the $n!$ orderings and run the stopping procedure for $T$. Its count $b_T(\Pi)$ of early neighbors is now a random variable:
\[
\E[b_T(\Pi)]=\frac{D(T)}{n!}.
\]
We can reuse the two inequalities proved within the induction. Dividing \eqref{eq:leafD} and \eqref{eq:splitD} by $n!$ gives
\begin{align*}
\text{leaf root:}\qquad &\E[b_T(\Pi)]\le\E[b_{T'}(\Pi)]+1,\\
\text{split branches:}\qquad &\E[b_T(\Pi)]\le\E[b_{T_1}(\Pi)]+\E[b_{T_2}(\Pi)]+1.
\end{align*}
These inequalities were established before applying induction in their respective cases. For a two-vertex tree, there are no early neighbors. Induction therefore gives
\begin{equation}\boxed{\E[b_T(\Pi)]\le t-2.}\label{eq:expected}\end{equation}
Indeed, the leaf case gives $(t-3)+1=t-2$, and the branch case gives $(t_1-2)+(t_2-2)+1=t-2$.

Now assume $\dd(G)>t-2$ and suppose $G$ contains no copy of the underlying unrooted tree $T$. Every run exhausts all vertices, so every neighbor of the first vertex $\Pi_0$ is early. But $\Pi_0$ is uniformly distributed over $V(G)$, so
\[
\E[b_T(\Pi)]
=\E[\deg(\Pi_0)]
=\frac1n\sum_{x\in V(G)}\deg(x)
=\dd(G)>t-2,
\]
a contradiction.

No independence among arrivals is needed. Nor do we assume that either rearrangement sends a uniformly random ordering to a uniformly random ordering. The injective counting arguments, including their one-per-ordering exceptions, are what justify the two expected-value inequalities.

\clearpage
\section{How the Recent Expositions Fit Together}\label{sec:comparisons}
The newer accounts average over vertex orderings, but do not all count the same objects. For this comparison, fix a graph $G$ on $n\ge2$ vertices and a rooted tree $T$ on $t\ge2$ vertices; no density assumption is needed.

\subsection*{Wood: Counting the Complement}
Wood counts all ordering--neighbor pairs in $M(G)$, and those already containing $T$ rooted at the first vertex in $R(T,r)$ \cite{wood}. Identifying a neighbor's recorded position with the vertex itself gives
\[
D(T)=|M(G)|-|R(T,r)|.
\]
Thus his bound is Proposition~\ref{prop:count} expressed using the complementary count.

\subsection*{Riordan--Scott: Neighbors after the Tree Appears}
Riordan and Scott call $v_0v_j$ \emph{$T$-jumping} when $T$ rooted at $v_0$ is already present \emph{before} $v_j$ arrives \cite[Section~1]{riordan-scott}. Let $J_T(\pi)$ count these edges. Let $s_T(\pi)$ be $1$ if our stopping arrival is a neighbor of $v_0$, and $0$ otherwise. Separating arrivals before, at, and after the stopping stage gives
\begin{equation}
\deg_G(v_0)=b_T(\pi)+s_T(\pi)+J_T(\pi).
\label{eq:early-jumping}
\end{equation}
If no copy appears, every neighbor is early and $s_T(\pi)=J_T(\pi)=0$, so the identity still holds. Jumping edges are \emph{not quite} the complement of early neighbors: a stopping neighbor belongs to neither count. Taking expectations over a uniformly random ordering $\Pi$ gives
\[
\E[J_T(\Pi)]\ge\dd(G)-(t-2)-1=\dd(G)-t+1,
\]
the bound they prove directly. They then apply it to a tree with one leaf removed: a jumping edge attaches that missing leaf at a new vertex.

For example, take our usual $G$ and $\pi=(b,e,a,d,c)$, with $T=P_2$. The run stops at $a$, so $b_T(\pi)=0$ and $s_T(\pi)=1$. The later neighbor $c$ contributes one jumping edge. Although our procedure stops, the retained ordering still specifies these later arrivals.

\subsection*{Frederickson: Cyclic Orderings and a Specified Edge}
A cyclic ordering places the vertex labels around a circle, identifying orderings that differ only by rotation. This does not assert that those vertices form a cycle in $G$.

Frederickson works in the directed setting and specifies both a root vertex and an incident tree edge, called the \emph{root arc} \cite{frederickson}. The corresponding edge of the host must be used in the copy; our containment test imposes no such requirement on the edge to the arriving neighbor. Reversing a cyclic ordering allows the root to move to the other endpoint of the specified edge without changing the relevant expectation. For a tree with at least two edges, one endpoint is not a leaf. This symmetry removes the need for a separate leaf-root reduction before splitting into smaller trees.

These viewpoints help explain why averaging is so effective: an individual reveal order may be unhelpful, but reversible rearrangements relate the counts across all orders. Our version keeps those relations at the level of individual early-neighbor occurrences, where they can be followed through the examples.

\clearpage
\section{An Application: Monochromatic Trees}\label{sec:ramsey}
For a tree $T$ and an integer $q\ge1$, the \emph{Ramsey number} $R(T;q)$ is the least $N$ such that every coloring of the edges of $K_N$ with $q$ colors contains a copy of $T$ whose edges all have one color. Such a copy is \emph{monochromatic}. The following classical implication of the \ES\ statement is discussed in \cite{bounded} and highlighted by Wood \cite[Theorem~2]{wood}.

\begin{corollary}\label{cor:ramsey}
For every tree $T$ on $t\ge2$ vertices and integer $q\ge1$,
\[
R(T;q)\le q(t-2)+2.
\]
\end{corollary}
\begin{proof}
Put $N=q(t-2)+2$. Some color occupies at least $\binom{N}{2}/q$ edges. The graph $H$ consisting of those edges and all $N$ vertices has
\[
\dd(H)\ge\frac{N-1}{q}=t-2+\frac1q>t-2.
\]
Since $N\ge t$, Theorem~\ref{thm:ES} gives a copy of $T$ in that color.
\end{proof}

For example, every blue--orange coloring of $K_6$ has a color with at least eight of its fifteen edges. That color's graph has average degree at least $16/6>2$, so it contains both four-vertex trees. The two copies need not use the same vertices.
\begin{pictureblock}
\panel{\begin{tikzpicture}[x=.85cm,y=.85cm]
\foreach \name/\xx/\yy in {a/0/1,b/1/2,c/2/2,d/3/1,e/2/0,f/1/0}{\coordinate (\name) at (\xx,\yy);}
\foreach \aa/\bb in {a/c,a/e,b/d,b/f,c/e,c/f,d/f}{\draw[accent,dashed,line width=.9pt] (\aa)--(\bb);}
\foreach \aa/\bb in {a/b,b/c,c/d,d/e,e/f,f/a,a/d,b/e}{\draw[chosen] (\aa)--(\bb);}
\foreach \name in {a,b,c,d,e,f}{\node[v] at (\name){$\name$};}
\end{tikzpicture}}{One coloring of $K_6$: eight solid blue edges and seven dashed orange edges.}
\hfill
\panel{\graphpic{a/0/1,b/1/2,c/2/2,d/3/1,e/2/0,f/1/0}{a/b,b/c,c/d,d/e,e/f,f/a,a/d,b/e}{a/b,a/f,a/d}}{The blue color class, with a star highlighted. The path $a-b-c-d$ is also present.}
\captionof{figure}{A monochromatic copy is a subgraph: edges of either color outside the selected tree do not matter. In the right panel, gray edges are other edges of the blue color class; orange edges have been omitted.}\label{fig:ramsey-example}
\end{pictureblock}

\clearpage
\appendix
\section{Basic Graph Theory}\label{app:graphs}
This appendix collects the terminology and elementary facts used in the paper.

\subsection{Graphs, Degrees, and Counting Each Edge Twice}
A \emph{finite simple undirected graph} consists of a finite set of vertices together with a set of edges, each joining two distinct vertices. There are no loops, no multiple edges between the same two vertices, and no directions on the edges. We write
\[
 G=(V(G),E(G)),\qquad n=|V(G)|,\qquad e(G)=|E(G)|.
\]
Here $|S|$ denotes the number of elements of a finite set $S$. The notation $uv\in E(G)$ means that $u$ and $v$ are joined by an edge; equivalently, they are \emph{adjacent}. Since edges are undirected, $uv$ and $vu$ name the same edge. Lines may cross in a drawing without making another vertex.

The \emph{degree} $\deg_G(v)$ of a vertex $v$ is its number of neighbors. When $G$ is understood, we write $\deg(v)$. Thus $0\le\deg(v)\le n-1$.  Moreover, if $v$ has degree zero, that means $v$ is \emph{isolated}, and if $v$ has degree $n-1$, that means $v$ is adjacent to every other vertex in $G$.
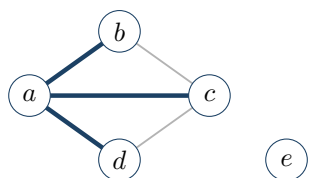
\begin{pictureblock}
\panel{\graphpic{a/0/1,b/1.4/2,c/2.8/1,d/1.4/0,e/4/0}{a/b,b/c,c/d,d/a,a/c}{a/b,a/c,a/d}}{The three highlighted edges contribute to $\deg(a)=3$.}
\hfill
\panel{\begin{tabular}{c|ccccc}vertex&$a$&$b$&$c$&$d$&$e$\\\hline degree&3&2&3&2&0\end{tabular}\par\medskip $3+2+3+2+0=10=2\cdot5$}{Five edges; ten incidences between vertices and edges.}
\captionof{figure}{Every edge has two ends, even though it is only one edge.}
\end{pictureblock}
What happens when we add together the degrees of all the vertices of $G$? In this sum, note that each edge $uv$ contributes one to $\deg(u)$ and one to $\deg(v)$, so it is counted exactly twice. This proves the \emph{handshaking identity}
\[\sum_{v\in V(G)}\deg(v)=2e(G).\]
For $n\ge1$, the \emph{average degree} is consequently
\[\dd(G)=\frac1n\sum_{v\in V(G)}\deg(v)=\frac{2e(G)}n.\]
In particular, if the average degree is $m$, the sum of the degrees is $mn$, and the number of edges is $mn/2$. The average $m$ need not be an integer. We will usually write it as $\dd(G)$ to keep it distinct from an edge count.

\subsection{Paths and Cycles}
A \emph{path} is a sequence of distinct vertices joined successively by edges. Its length is its number of edges, one fewer than its number of vertices. The notation $P_t$ denotes a path on $t$ vertices. A single vertex is also a path, of length zero.

A \emph{cycle} consists of at least three distinct vertices joined successively by edges, with an additional edge joining the last vertex to the first. Its length is both its number of edges and its number of vertices. The notation $C_t$ denotes a cycle on $t$ vertices. When we describe a path or cycle inside a larger graph, we may ignore other edges between its vertices.
\begin{pictureblock}
\panel{\graphpic{a/0/0,b/1/.6,c/2/0,d/3/.6}{a/b,b/c,c/d}{a/b,b/c,c/d}}{The path $a-b-c-d$: four vertices, length three.}
\hfill
\panel{\graphpic{a/0/0,b/0/1.5,c/1.5/1.5,d/1.5/0}{a/b,b/c,c/d,d/a}{a/b,b/c,c/d,d/a}}{The cycle $a-b-c-d-a$: four vertices, length four.}
\captionof{figure}{A path does not repeat a vertex; a cycle closes back at its starting vertex.}
\end{pictureblock}

\subsection{Connected Graphs}
A graph is \emph{connected} if every two vertices can be joined by a path. Otherwise it is \emph{disconnected}. A \emph{connected component} is a maximal connected part of the graph: its vertices can reach one another by paths, but cannot reach any vertex outside it.
\begin{pictureblock}
\panel{\graphpic{a/0/0,b/1/1,c/2/0,d/3/1,e/4/0}{a/b,b/c,a/c,c/d,d/e}{a/b,b/c,a/c,c/d,d/e}}{Connected: the edge $cd$ joins the two parts.}
\hfill
\panel{\graphpic{a/0/0,b/1/1,c/2/0,d/3/1,e/4/0}{a/b,b/c,a/c,d/e}{a/b,b/c,a/c,d/e}}{Disconnected: there is no path from $a$ to $e$.}
\captionof{figure}{The graph on the right has two connected components. Connectivity concerns paths, not whether every pair is adjacent.}
\end{pictureblock}

\subsection{Trees}
A \emph{tree} is a nonempty connected graph with no cycle. A vertex of degree one is called a \emph{leaf}. A \emph{star} $K_{1,s}$ is a tree consisting of one center joined to $s$ leaves, with no other edges.
\begin{pictureblock}
\panel{\graphpic{a/0/1,b/1/1,c/2/2,d/2/0,e/3/0}{a/b,b/c,b/d,d/e}{a/b,b/c,b/d,d/e}}{A tree: connected, with no cycle. Its leaves are $a,c,e$.}
\hfill
\panel{\graphpic{u/1/1,a/0/0,b/2/0,c/1/2}{u/a,u/b,u/c}{u/a,u/b,u/c}}{The star $K_{1,3}$: another tree.}
\captionof{figure}{A tree may branch. It need not look like a path or a star.}
\end{pictureblock}
A tree on $t$ vertices has exactly $t-1$ edges. One way to see this is to repeatedly remove a leaf, along with its incident edge, until one vertex remains. Every nontrivial tree has a leaf: an endpoint of a longest path has no possible additional neighbor without extending the path or creating a cycle. Removing a leaf preserves connectedness among the remaining vertices and cannot create a cycle, so the process continues. It removes $t-1$ vertices and exactly $t-1$ edges.

A tree is also \emph{bipartite}: its vertices can be colored with two colors so that each edge joins opposite colors. To see this, choose a root and color each vertex according to whether its distance from the root is even or odd. There is a unique path from the root to each vertex: connectedness gives a path, and two distinct paths would create a cycle. Its length defines the distance, and every edge of a tree joins a parent to a child one step farther from the root.

\subsection{Subgraphs and Containment}
A graph $G$ \emph{contains} a tree $T$ if there exists an isomorphic copy of $T$ as a subgraph of $G$. Extra edges of $G$ are permitted, and vertices of $G$ may be unused. We are asking for a subgraph, not necessarily an \emph{induced} subgraph. An induced subgraph on a chosen set keeps \emph{all} the edges between those vertices.
\begin{pictureblock}
\panel{\graphpic{a/0/0,b/1.3/1.2,c/2.6/0}{a/b,b/c,a/c}{a/b,b/c}}{A triangle contains the path $a-b-c$.}
\hfill
\panel{\graphpic{a/0/0,b/1.3/1.2,c/2.6/0}{a/b,b/c}{a/b,b/c}}{The selected path has only the two blue edges.}
\captionof{figure}{Containment does not require the chosen vertices to induce a tree.}
\end{pictureblock}
More generally, a \emph{subgraph} is obtained by selecting some vertices of a graph and some of its edges whose endpoints are selected. Two graphs are \emph{isomorphic} if their vertices can be put in a one-to-one correspondence preserving adjacency in both directions; informally, they have the same shape, regardless of their labels or drawings.

The \emph{complete graph} $K_s$ has $s$ vertices and every possible edge between distinct vertices. An \emph{independent set} has no edges between its vertices. These describe opposite extremes of adjacency, not necessarily opposite extremes of connectivity in a larger graph.

\section*{Statement on the Use of AI}
The underlying proof of the Erd\H{o}s--S\'os theorem presented here was discovered by GPT-6 Astra; its original form and subsequent human expositions are cited in Section~\ref{sec:history}. This article was developed through an extended dialogue between the author and OpenAI's ChatGPT. ChatGPT assisted with discussing the mathematics, proposing and evaluating alternative organizations of the proof, drafting and revising prose, producing LaTeX and TikZ code, checking internal consistency, and locating references. The essential changes in this paper include the reveal-and-stop formulation, the division of the inductive counts into explicit sets, and the examples and diagrams, all of which originated with the author, although they were refined through this collaborative process. The author reviewed everything that the AI produced and takes full responsibility for the mathematical arguments, exposition, and references.

\end{document}